\documentclass[a4paper]{article}
\usepackage{graphicx}
\usepackage{latexsym}
\usepackage{epsfig,epstopdf}
\usepackage{amsmath,amssymb,amsfonts}
 \usepackage{graphicx,color}
\usepackage{float}
\renewcommand{\baselinestretch}{1.2}
\newtheorem{prethm}{{\bf Theorem}}

\newenvironment{thm}{\begin{prethm}{\hspace{-0.5
               em}{\bf.}}}{\end{prethm}}

\newtheorem{prepro}[prethm]{Proposition}
\newenvironment{pro}{\begin{prepro}{\hspace{-0.5
               em}{\bf.}}}{\end{prepro}}

\newtheorem{prelem}[prethm]{Lemma}

\newtheorem{precor}[prethm]{Corollary}
\newenvironment{cor}{\begin{precor}{\hspace{-0.5
               em}{\bf.}}}{\end{precor}}
\newtheorem{precon}[prethm]{Conjecture}
\newenvironment{con}{\begin{precon}{\hspace{-0.5
               em}{\bf.}}}{\end{precon}}

\newtheorem{preremark}{{\bf Remark}}

\newenvironment{rem}{\begin{preremark}\em{\hspace{-0.5
              em}{\bf.}}}{\end{preremark}}

\newtheorem{preexample}{{\bf Example}}

\newtheorem{preproblem}{{\bf problem}}

\newtheorem{preproof}{{\bf Proof.}}

\newenvironment{proof}[1]{\begin{preproof}{\rm
               #1}\hfill{$\Box$}}{\end{preproof}}

\newcommand{\1}{\mathbf{1}}
\renewcommand{\thefootnote}

\begin{document}
\title{Upper bounds on the number of perfect matchings\\ 
and directed 2-factors in graphs with given number of vertices and edges}
\author{{\normalsize{\sc M. Aaghabali ${}^{3}$}, {\sc S. Akbari
${}^{1,4}$},\, {\sc S. Friedland ${}^{2}$},  {\sc K. Markstr\"om ${}^{5}$},
{\sc Z. Tajfirouz ${}^{3},$}}\\{\footnotesize{${}^{1}$\it Department of
Mathematical Sciences, Sharif University of Technology,Teheran, Iran}}\\
{\footnotesize}{\footnotesize{${}^{2}$\it Department of Mathematics, Statistics and Computer Science, University of Illinois at Chicago, USA }}\\
{\footnotesize}{\footnotesize{${}^{5}$\it Department of Mathematics and Mathematical Statistics, Ume{\aa} University,  Sweden }}\\
{\footnotesize}{\footnotesize{${}^{3}$\it Department of Mathematics, University of Zanjan, Iran}}\\{\footnotesize} {\footnotesize{${}^{4}$\it Institute
for Studies in Theoretical Physics and Mathematics
{\rm(IPM)}}}\vspace{-2mm}\\{\footnotesize{\it Tehran, Iran}}
\\{\footnotesize{$\mathsf{maghabali@znu.ac.ir}$\quad\quad
$\mathsf{s\_akbari@sharif.edu}$\quad\quad$\mathsf{friedlan@uic.edu
}$}\,\,\quad}\\{\footnotesize
$\mathsf{ klas.markstrom@math.umu.se}$\,\,\quad}
{\footnotesize
$\mathsf{z.tajfirouz@yahoo.com}$\,\,\quad}
\thanks {The third author was supported by NSF grant DMS-1216393.
}}
\date{}
\maketitle
\begin{quote}
{\small \hfill{\rule{13.3cm}{.1mm}\hskip2cm} \textbf{Abstract}
	 We give an upper bound on the number of perfect matchings in simple graphs with a given number of vertices and edges.
 	We apply this result to give an upper bound on the number of 2-factors in a directed complete bipartite balanced graph
 	on $2n$ vertices.  The upper bound is sharp for $n$ even.  For $n$ odd we state a conjecture on a sharp upper bound.
\vspace{1mm} {\renewcommand{\baselinestretch}{1}
\parskip = 0 mm

}

}

\vspace{-3mm}\hfill{\rule{13.3cm}{.1mm}\hskip2cm}
\end{quote}
\section{\bf  Introduction}
 Let $G=(V,E)$ be a simple graph with adjacency matrix $A$. We assume that $m:=|E|$ and $2n:=|V|$ is even.
 Let ${\rm deg}(v)$ be the degree of the vertex $v\in V$. We denote by ${\rm perfmat}\; G$ the number of perfect matchings in $G$.
 Recall the following upper bound for the number of perfect matchings in $G$, see e.g.  \cite{AF08} for a proof,
 \begin{equation}\label{afubnd}
	 {\rm perfmat}\; G\le \prod_{v\in V} ({\rm deg}(v)!)^{\frac{1}{2{\rm deg}(v)}}.
 \end{equation}
 Equality holds if and only if $G$ is a disjoint union of complete regular bipartite graphs.
 Denote by $K_{n,n}$ the complete bipartite graph on $2n$ vertices.

 The first major result of this paper is to give a sharp upper bound of the right-hand side of the above inequality.
 Assume that $m\ge 2n\ge 2$.  Denote
 \begin{equation}\label{defomegmn}
	 \omega(2n,m):=(\lfloor \frac{m}{n}\rfloor !)^{\frac{n-\alpha}{\lfloor \frac{m}{n}\rfloor }}
  (\lceil \frac{m}{n}\rceil !)^{\frac{\alpha}{\lceil \frac{m}{n}\rceil }}, \; \alpha:=m-n\lfloor \frac{m}{n}\rfloor.
 \end{equation}
 (Where  $0!=0$. )   Denote by $\mathcal{G}(2n,m)$ the set of all simple graphs with $2n$ vertices and $m$ edges.
 We show that
 \begin{equation}\label{propommn}
	 \prod_{v\in V} ({\rm deg}(v)!)^{\frac{1}{2{\rm deg}(v)}} \le \omega(2n,m) \text{ ~for each }~ G\in \mathcal{G}(2n,m).
 \end{equation}
 Equality holds if and only if $G$ is almost regular.
 \begin{equation}\label{balcond}
	 |\deg(u)-\deg(v)|\le 1 \text{~ for all }~ u,v\in V,
 \end{equation}
 i.e. there exists a positive integer $k$ such that the degree of each vertex of $G$ is $k$ or $k-1$.
 Hence
 \begin{equation}\label{omegubnd}
	 {\rm perfmat}\; G\le \omega(2n,m) \text{ ~for each }~ G\in \mathcal{G}(2n,m).
 \end{equation}
 Equality holds if and only if and only if $G=G^{\star}_{2n,m}$, where $G^\star_{2n,m}$  is a disjoint union of $\ell_1 \ge 1$ copies of
 $K_{n_1,n_1}$ and $\ell_2\ge 0$ copies of $K_{n_1+1,n_1+1}$.  (So $G^{\star}=\ell_1 K_{n_1,n_1}\cup \ell_2 K_{n_1+1,n_1+1}$ 
 and  $n=\ell_1n_1+\ell_2(n_1+1)$ and $m= \ell_1n_1^2+\ell_2(n_1+1)^2$.)

 The second major result of this paper deals with an upper bound on the permanent of tournament matrices corresponding to $K_{n,n}$.
 Let $\mathcal{D}_{n}$ be the set of all digraphs  obtained from $K_{n,n}$ by assigning a direction to each edge.
 For each $D\in \mathcal{D}_{n}$ let $A_D$ be the $(0,1)$ adjacency matrix of order $2n$.  Then  ${\rm per} (A_D)$ is the number of
 directed 2-factors of $D$.  We show that
 \begin{equation}\label{tournupbds}
	 {\rm per} (A_D)\le \omega^2(4n,n^2) \text{~ for~ each }~ D\in\mathcal{D}_n.
 \end{equation}
 For $n$ even the above inequality is sharp.  For an odd integer we conjecture the sharp inequality 
 \begin{equation}\label{tournupbdsconj}
	 {\rm per} (A_D)\le  \frac{n+1}{2} D_{ \frac{n+1}{2}}(\frac{n-1}{2})!^3 \text{ ~for odd}~ n \ne 5  \text{~ and each }~  D\in\mathcal{D}_n.
 \end{equation}
 Here $D_p$ stands for the number of derangements of $\{1,\ldots,p\}$.  We give some computational support for this conjecture.

We now summarize briefly the contents of the paper. In \S2  we give some preliminary results containing basic definitions, notation and estimates for 
the number of perfect matchings in an arbitrary graph. We present some lower and upper bounds for the maximal number of $2$-factors of tournaments. 
In \S3 we prove a key inequality of our paper.  In \S4 we apply this inequality to get an upper bound for the number of perfect matchings in a graph with 
a given number of vertices and edges.  Also, for even $n$,  we obtain an orientation for the edges of the complete bipartite graph $K_{n,n}$ in such a 
way that it has a maximum number of $2$-factors.  For odd $n$ we conjecture that the maximum number of $2$-factors occurs in an explicitly 
given special complete bipartite digraphs. In the last section  we state some computational results for the permanent of tournaments of order up to $10$ 
and the number of perfect matchings in graphs with a small number of vertices, or with the number of edges close to the half of number of vertices.

\section{\bf  Preliminary results}

By the \textit{permanent} of an $n\times n$ matrix $A=[a_{ij}]$ over a commutative ring we mean  $${\rm per} (A)=
\sum_\sigma a_{1\sigma(1)}a_{2\sigma(2)}\dots a_{n\sigma(n)},$$
 where the summation is over all permutations of  $\{1,\dots,n\}$.

 Suppose that $D=(V,Arc)$ is a simple directed graph, where $V=\{v_1,\dots,v_n\}$.
 (That is each directed edge $(v_i,v_j)$ appears at most once.  We allow loops $(v_i,v_i)$.)
 Let ${\rm deg}^+(v)$ and ${\rm deg}^-(v)$ denote the outdegree and indegree of the vertex $v\in V$, respectively. The \textit{adjacency matrix} of 
$D$ is an $n\times n$ matrix $A_D=[a_{ij}]_{i,j=1}^n$ indexed by the vertex set, where $a_{ij}=1$ when  $(v_i,v_j)\in Arc$ and $a_{ij}=0$, otherwise. 
A disjoint union of directed cycles in $D$ is called $2$\textit{-factor} if it covers all vertices of $D$.  We allow loops and $2$-cycles $v_i\to v_j\to v_i$ for $i<j$.
Then ${\rm per} (A_D)$ counts the number of $2$-factors of $D$.  Denote by ${\cal D}(n,m)$ the set of simple digraphs $D=(V, Arc)$ with $n=|V|$ vertices
 and $m=|Arc|$ directed edges.

Let $G=(A,E)$ be a simple graph.  Denote by $\mathcal{D}(G)$ the set of all directed graphs $D=(V,Arc)$ obtained from $G$ by orienting each edge of $E$. 
So $|Arc|=|E|$.  Each $A_D=[a_{ij}(D)], D\in \mathcal{D}(G)$ is a combinatorial 
skew symmetric matrix, i.e $a_{ij}(D)a_{ji}(D)=0$ for all $i,j\in [n]:=\{1,\ldots,n\}$.  Note that each $2$-factor of $D$ consists of cycles of length at least $3$.  
We denote $\mathcal{D}_n:=\mathcal{D}(K_{n,n})$.

 It is natural to ask for an assignment of directions to the edges of an undirected graph which results in a directed graph with the maximum number of  $2$-factors. 
The aim of this paper is to answer this question in some special cases.

 Let
  \begin{equation}\label{defmu2nm}
  \mu(2n,m):=\max_{G\in \mathcal{G}(2n,m)} {\rm perfmat}\; G ={\rm perfmat}\; G_{2n,m}^\star
  \end{equation}
  and
 \begin{equation}\label{defnu2nm}
  \nu(n,m):=\max_{D\in \mathcal{D}(n,m)} {\rm per}(A_D) ={\rm per}\; (A_{D^\star_{n,m}}),
  \end{equation}

   The first major result of this paper is the following inequality.
  \begin{equation}\label{muineq}
  \mu(2n,m)\le \omega(2n,m).
  \end{equation}
  Furthermore,  \eqref{muineq} is strict unless $\mathcal{G}(2n,m)$ contains $\ell_1K_{n_1,n_1}\cup \ell_2 K_{n_1+1,n_1+1}$.

  We say that $G=(V,E)$ is \emph{unbalanced} if there exist two vertices $u,v\in V$ such that ${\rm deg}(v)\geq{\rm deg}(u)+2$, and all neighbors of $u$
  that are different from $v$ are neighbors $v$.  Otherwise $G$ is called a \emph{balanced} graph.
  We show that each maximal $G^\star_{2n,m}$ can be chosen balanced.

  Let $k\in [n]:=\{1,\ldots,n\}$.
  Denote by $\mathcal{B}(2n,k)\subset \mathcal{G}$ the set of all simple $k$-regular bipartite graphs on $2n$ vertices.

  Clearly, $\mu(2n,n)=1$.  In that case \eqref{muineq} is sharp.
  Assume that
  \begin{equation}\label{mncond}
  1 < \frac{m}{n}\le n.
  \end{equation}
  Let $k:=\lceil \frac{m}{n}\rceil$ and $\alpha=m-n\lfloor \frac{m}{n}\rfloor$.
   Note that if $\alpha=0$ (or equivalently $\frac{m}{n}\in\mathbb{N}$), then any bipartite $k$-regular graph on $2n$ vertices has  $m$ edges, so  
 $\mathcal{B}(2n,k)\subset\mathcal{G}(2n,m)$. Thus consider  $H'\in \mathcal{B}(2n,k)$ and use the van der Waerden permanent inequality to yield
   \begin{equation}\label{mulowbnd1}
  {\rm perfmat}\; H'> (\frac{k}{e})^{n}.
  \end{equation}

  If $\alpha>0$ (or equivalently $ \frac{m}{n}\notin\mathbb{N}$), consider $H\in \mathcal{B}(2n,k)$. K\"onig's theorem yields that 
 each $H\in \mathcal{B}(2n,k)$ is $k$-edge colorable.  (Note that each color represents a perfect matching.)
  Color the edges of $H$ in $k$-colors and delete $n-\alpha$ edges colored in the first color to obtain a graph $G\in\mathcal{G}(2n,m)$.
 Then $G$ contains a subgraph $F\in \mathcal{B}(2n,k-1)$ which is colored in colors $2,\ldots,k$.
  Use the van der Waerden permanent inequality or its lower approximation \cite{Fri79} to deduce
  \begin{equation}\label{mulowbnd}
  {\rm perfmat}\; G \ge {\rm perfmat}\; F> (\frac{k-1}{e})^{n}.
  \end{equation}

  If $\frac{m}{n}$ is of order $n$ it is easy to show using Stirling's inequality
  \begin{equation}\label{stirineq}
  p!=\sqrt{2\pi p} (\frac{p}{e})^p e^{r_p}, \quad \frac{1}{12p+1}< r_p <\frac{1}{12p}
  \end{equation}
 that the above lower bound is of the same order  as the upper bound \eqref{muineq}.\\

Here we give the bipartite transformation of a directed graph which is extremely useful in our studies. Let $D=(V, Arc)$ be a directed graph, where 
 $V=\{v_1,\ldots,v_n\}$.
Denote by $B(D)$ the bipartite graph with partite sets $ V'=\{v_1',\ldots,v_n'\}, V''=\{v_1'',\ldots,v_n''\}$.
 Then $(v_i',v_j'')$ is an edge in $B(D)$,   if and only if $(v_i, v_j)$ is a directed edge in $D$.  Assume that 
$C=A_D$ is the adjacency matrix of $D$.  Then
$$S=\left[\begin{array}{ll}0&C\\C^T&0\end{array}\right],$$ 
is the adjacency matrix of $B(D)$. It is not hard to see that ${\rm per}(S)={\rm perfmat}^2(B(D))$.  So,
 \begin{equation}\label{perbr}
 {\rm per} (S)={\rm per}(C){\rm per} (C^T)={\rm  per}(A_D){\rm per} (A_D^T)={\rm per} ^2(A_D).
 \end{equation}
 Thus we deduce that the number of perfect matchings of $B(D)$ is equal to the number of $2$-factors of $D$.

Observe that ${\rm deg} (v_i')={\rm deg}^+(v_i), {\rm deg} (v_i'')={\rm deg}^-(v_i)$.
We say a directed graph $D$ is balanced if $B(D)$ is balanced.

Recall that  a $(0,1)$-matrix $T=[t_{ij}]_{i,j=1}^n$ is called a \textit{tournament matrix}  if $t_{ii}=0$ and $t_{ij}+t_{ji}=1$  for $i\ne j$. 
This means that $T$ is the adjacency matrix of a directed 
complete graph of order $n$, i.e a \textit{tournament} of order $n$. Thus  $B(T)\in\mathcal{G}(2n,\frac{n(n-1)}{2})$.  
Denote by $\mathcal{T}(n)\subset \mathcal{G}(2n,\frac{n(n-1)}{2})$ the set of all bipartite representations of tournaments of order $n$.

Let
\begin{equation}\label{deftau2nm}
	\tau(n):=\max_{G\in \mathcal{T}(n)} {\rm perfmat}\; G ={\rm perfmat}\; T_{n}^\star.
\end{equation}
Clearly $\tau(n)\le \mu(2n,\frac{n(n-1)}{2})$.   So we can use the upper bound \eqref{muineq}.
Note that here $\alpha=0$ if $n$ is odd and $\alpha=\lceil \frac{n-1}{2}\rceil=\frac{n}{2}$ if $n$ is even.
Suppose that $n$ is odd.  It is well known that there exists a tournament $T$ such that each player wins half of the games,
in which case  the corresponding bipartite graph is in $\mathcal{B}(2n,\frac{n-1}{2})$.  For this graph $B(T)$ one can use the lower permanent bound given in
[4, p.99].  Combine this observation with \eqref{muineq} to deduce
\begin{equation}\label{lowtaubd}
	\sqrt{e}(\frac{n-1}{2e})^n \leq\tau(n)\le ((\frac{n-1}{2})!)^{\frac{2n}{n-1}}  .
\end{equation}
Again, Stirling's inequality implies that the upper and the lower bounds for $\tau(n)$ are of the same order.

In the case of  even $n$ one can consider a tournament $T$ with vertex set $\{v_1,\dots,v_n\}$ such that for every $i\leq \frac{n}{2}$, 
$(v_i,v_{i+1}),\dots,(v_i,v_{i+\frac{n}{2}})\in Arc$ and for every $i>\frac{n}{2}$, $(v_i,v_{i+1}),\dots,(v_i,v_{i+\frac{n}{2}-1})\in Arc$, 
where indices are computed in mod $n.$ In this tournament for any $i\leq\frac{n}{2}$, ${\rm deg}^+(v_i)=\frac{n}{2}$, and for any 
$i>\frac{n}{2}$, ${\rm deg}^-(v_i)=\frac{n}{2}.$ Therefore, the corresponding bipartite graph $B(T)$ is an $(\frac{n}{2},\frac{n}{2}-1)$-
near regular graph in  which 
$${\rm deg}(v'_1)=\dots={\rm deg}(v'_{\frac{n}{2}})={\rm deg}(v''_{\frac{n}{2}+1})=\dots={\rm deg}(v''_n)=\frac{n}{2}$$
 and $${\rm deg}(v''_1)=\dots={\rm deg}(v''_{\frac{n}{2}})={\rm deg}(v'_{\frac{n}{2}+1})=\dots={\rm deg}(v'_n)=\frac{n}{2}-1.$$ 
From the construction of $T$ it is clear that $\{v'_1,v''_{\frac{n}{2}+1}\},\dots,\{v'_{\frac{n}{2}},v''_n\}$ are edges of $B(T)$. 
By deleting these edges we obtain a subgraph $F\in \mathcal{B}(2n,\lfloor \frac{n-1}{2}\rfloor)$. Again, using the lower permanent 
bound given in [4, p.99] we have
\begin{equation}\label{taulowbnd}
	{\rm perfmat}\; B(T) \ge {\rm perfmat}\; F\ge \sqrt{e}(\frac{n-2}{2e})^{n}.
\end{equation}
Use the upper bound \eqref{muineq} to obtain
\begin{equation}\label{lowtaubdeven}
	\sqrt{e}(\frac{n-2}{2e})^n \leq\tau(n)\le ((\frac{n-2}{2})!)^{\frac{2n-2}{n-2}}(\frac{n}{2}).
\end{equation}

 \section{A fundamental inequality}
 The following result is due to I.M. Wanless \cite{Wan99}.  (It is explicitly mentioned in the beginning of the proof of Lemma 1.)
 For reader's convenience we include its short proof.   
 \begin{thm}\label{fundineqthm}
 For $p\in\mathbb{N}$ the sequence $a_p:=\frac{(p!)^\frac{1}{p}}{(p+1)!^{\frac{1}{p+1}}}$ is a strictly increasing sequence.

\end{thm}
\begin{proof}
 {We want to prove the following inequality
 \begin{equation}\label{fundineq}
 \frac{(p!)^\frac{1}{p}}{(p+1)!^{\frac{1}{p+1}}}<\frac{(p+1)!^\frac{1}{p+1}}{(p+2)!^{\frac{1}{p+2}}} \textrm{ for } p=1,2,\ldots .
 \end{equation}
Use the equality $(p+1)!=(p+1)p!$ and $(p+2)!=(p+2)(p+1)p!$ we find that

$$p!^{\frac{2}{p(p+1)(p+2)}}\leq \frac{(p+1)^{\frac{p+3}{(p+1)(p+2)}}}{(p+2)^{\frac{1}{p+2}}}$$
This implies that $$(p!)^2(p+2)^{p(p+1)}\leq (p+1)^{p(p+3)}$$
By the arithmetic-geometric mean inequality we have $$\sqrt[p]{p!}\leq \frac{1+\dots+p}{p}=\frac{p+1}{2}$$
Thus it is suffices to prove that $$\frac{p+2}{p+1}\leq 2^{\frac{2}{p+1}}$$
Clearly, for $x\geq 0,~~~~~~1+x\leq e^x$. If $x=\frac{1}{p+1}$, then we have $\frac{p+2}{p+1}\leq e^{\frac{1}{p+1}}\leq 2^{\frac{2}{p+1}}$ and the proof is complete.}
\end{proof}
 \begin{cor}\label{fundineqlem}
  Let $p,q$ be nonnegative integers.  If $p\le q-2$ then
 \begin{equation}\label{fundineq3}
 (p!)^{\frac{1}{p}}(q!)^{\frac{1}{q}} <((p+1)!)^{\frac{1}{p+1}}((q-1)!)^{\frac{1}{q-1}}.
 \end{equation}
 In particular
 \begin{equation}\label{fundineq4}
 (p!)^{\frac{1}{p}}(q!)^{\frac{1}{q}} \le (\lfloor \frac{p+q}{2}\rfloor !)^{\frac{1}{\lfloor \frac{p+q}{2}\rfloor}}
 (\lceil \frac{p+q}{2}\rceil !)^{\frac{1}{\lceil \frac{p+q}{2}\rceil}}.
 \end{equation}
 Equality holds if and only if the multiset $\{p,q\}$ equals to $\{\lfloor \frac{p+q}{2}\rfloor,\lceil \frac{p+q}{2}\rceil \}$.
 \end{cor}
 Apply the above corollary to deduce.
 \begin{thm}\label{fundprodin} Let $k\ge 2$ be an integer, and assume that $p_1,\ldots,p_k$ are positive integers.  Let
 $M=p_1+\cdots+p_k$.  Define 
 \begin{equation}\label{thetakm}
  \theta(k,M):=(\lfloor \frac{M}{k}\rfloor !)^{\frac{k-\beta}{\lfloor \frac{M}{k}\rfloor}}
  (\lceil \frac{M}{k}\rceil !)^{\frac{\beta}{\lceil \frac{M}{k}\rceil}}, \; \beta:=M-k\lfloor \frac{M}{k}\rfloor.
 \end{equation}
 Then 
 \begin{equation}\label{fundineq5}
 \prod_{i=1}^k (p_i!)^{\frac{1}{p_i}}\le \theta(k,M).
 \end{equation}
 Equality holds if and only if $|p_i-p_j|\le 1$ for all integers $i,j=1,\ldots,k$.
 \end{thm}

 \section{Main inequalities for matchings}
\noindent
\begin{thm}\label{matchthmin}
	Let $m$ and $n$ be two positive integers, such that $n\le m\le {2n \choose 2}$.  Assume that $\omega(2n,m)$ is defined by \eqref{defomegmn}.   
  Let $G=(V,E)\in {\cal G}(2n,m)$.  Then inequality \eqref{propommn} holds.  Equality holds if and only if $G$ is almost regular.  	
	Let $\mu(2n,m)$ be defined as in \eqref{defmu2nm}. Then $\mu(2n,m)\le \omega(2n,m)$. Equality holds if and only if $\mathcal{G}(2n,m)$ 
contains a graph $G^{\star}$ which is a disjoint union 
	of $\ell_1 \ge 1$ copies of $K_{n_1,n_1}$ and 	$\ell_2\ge 0$ copies of $K_{n_1+1,n_1+1}$.  I.e. 
$n=\ell_1 n_1+\ell_2(n_1+1), m=\ell_1n_1^2+\ell_2(n_1+1)^2$.
	Furthermore, for these values of $n$ and $m$ a maximal graph $G^\star_{2n,m}$ in \eqref{defmu2nm} is unique and equal to 
$\ell_1K_{n_1,n_1}\cup\ell_2 K_{n_1+1,n_1+1}$.
 \end{thm}
 \begin{proof}
 {Let $G\in \mathcal{G}(2n,m)$. 
 Combine \eqref{afubnd} with Theorem \ref{fundprodin} to deduce \eqref{propommn}.  The equality case in Theorem  \ref{fundprodin} yields that equality
 in \eqref{propommn} holds if and only if $G$ is an almost regular graph.   
 The definition of $\mu(2n,m)$ and the inequality \eqref{propommn} yields the inequality $\mu(2n,m)\le \omega(2n,m)$.
 Equality holds if and only if there exists an almost regular graph $G^\star_{2n,m}$ for which equality holds in the Alon-Friedland  upper bound \eqref{afubnd}.
 So $G^\star_{2n,m}$ is a union of complete bipartite graphs.  Hence $G^\star_{2n,m}=\ell_1 K_{n_1,n_1}\cup\ell_2 K_{n_1+1,n_1+1}$.}
 \end{proof}
One might hope that for each integer $m$, $n\le m\le {2n\choose 2}$ each maximal graph is almost regular, however as we shall see 
in Section \ref{comp} this is not always the case.
 Let $\mathcal{BAL}(2n,m)$ denote the subset of $\mathcal{G}(2n,m)$ containing all balanced graphs with $2n$ vertices and $m$ edges.
 The following theorem implies that for each $n,m$ as above there exists a balanced maximal $G^\star_{2n,m}$.
 \begin{thm}\label{blangthm}
 Let $G\in\mathcal{G}(2n,m)$.  Then there exists $G_0\in\mathcal{BAL}(2n,m)$ such that ${\rm perfmat}\; G\le {\rm perfmat}\;G_0$.
 Suppose furthermore that $G$ is bipartite.  Then $G_0$ can be chosen bipartite.

 \end{thm}
 \begin{proof}
 {If $G$ is a balanced graph, then there is nothing to prove. So we assume that  $G$ is a graph which is not balanced.
 Then there exist $i,j\in[2n]$ such that ${\rm deg}(v_i)\ge {\rm deg}(v_j)+2$.  Furthermore, each neighbor of $v_j$ is a neighbor of $v_i$.
 Denote by $N(v_j)$ all neighbors of $v_j$.  If $N(v_j)=\emptyset$ then choose $G_0$ to be any bipartite balanced graph in $\mathcal{G}(2n,m)$.
 So assume that $N(v_j)\ne \emptyset$.  Hence, if $G=(V_1\cup V_2,E)$ is bipartite then $v_i,v_j\in V_p$ for some $p\in\{1,2\}$.
 Let $v_k$ be a neighbor  of $v_i$ which is not a neighbor of $v_j$.  Let $\mu',\mu(i,k), \mu(\{i,k\},\{j,l\}), v_l\in N(v_j)$ be the number
 of all perfect matchings of $G$ which do not contain the edge $(v_i,v_k)$, which contain the edge $(v_i,v_k)$, and which contain the edges $(v_i,v_k), (v_j,v_l)$.  So
 \[{\rm perfmat}\; G=\mu'+\mu(i,k), \quad \mu(i,k)=\sum_{v_l\in N(v_j)} \mu(\{i,k\},\{j,l\}).\]
 Let $G_1$ be a graph obtained by deleting the edge $(v_i,v_k)$ and adding the edge $(v_j,v_k)$.  Note that if $G$ is bipartite so is $G_1$.
 Clearly, any perfect matching in $G$ which does not contain $(v_i,v_k)$ is also a perfect matching in $G_1$.  Now to any perfect matching in $G$ containing
 the pairs $(v_i,v_k),(v_j,v_l)$ corresponds a unique perfect matching containing $(v_i,v_l),(v_j,v_k)$.  In view of the above equality for ${\rm perfmat}\; G$
 we deduce that ${\rm perfmat}\; G \le {\rm perfmat}\; G_1$.  If $G_1$ is balanced we are done.  If $G_1$ is not balanced we continue this process.
 Note when comparing the degree sequence of $G$ to $G_1$ we see that we changed only two degrees 
 ${\rm deg}_G(v_i)={\rm deg}_{G_1}(v_i) +1, {\rm deg}_G(v_j)={\rm deg}_{G_1}(v_j) -1$.  
In other words, the degree sequence of $G$ strictly majorizes the degree sequence of $G_1$.  Hence this process must stop at a balanced 
 graph $G_0\in {\cal G}(2n,m)$.  }
 \end{proof}
 \begin{cor} For any positive integers $n,m$, $m\le {2n \choose 2}$ there exists a  balanced graph $G^\star_{2n,m}$ that satisfies \eqref{defmu2nm}.
 \end{cor}

 Denote by $\mathcal{BALD}(n,m)$ the subset of all balanced digraphs in ${\cal D}(n,m)$.
 \begin{pro} Let $n,m$ be two positive integers.  Assume that $n\le m\le n^2$.  Then
  $\nu(n,m)=\max_{D\in \mathcal{BALD}(n,m)} {\rm per}(A_D).$
 \end{pro}
 \begin{proof} {Let $D=D^\star_{n,m}$ be a maximal digraph satisfying \eqref{defnu2nm}.  Since $m\ge n$ we deduce that ${\rm per} (A_D)\ge 1$.
 Let $B(D)$ be the corresponding bipartite graph in ${\cal G}(2n,m)$.  So ${\rm perfmat}\; B(D)={\rm per} (A_D)\ge 1$.  Theorem \ref{blangthm}
 yields that there exists a balanced bipartite graph $G_0\in{\cal G}(2n,m)$ such that ${\rm perfmat}\; B(D)\le {\rm perfmat}\; G_0$.  Since $G_0$ has at least 
one perfect matching it follows that 
 each group of the vertices of $G_0$ has $n$ vertices.  Hence $G_0=B(D_0)$ for some $D_0\in {\cal BALD}(n,m)$.  So ${\rm per}(A_ {D_0})=\nu(n,m)$. 
 }
 
 \end{proof}
 \begin{rem}
By \eqref{perbr} there is a straightforward relation between the number of $2$-factors of a directed graph $D$ and the permanent of the adjacency matrix 
of its bipartite transformation $B(D)$. On the other hand, in bipartite graphs the permanent of the adjacency matrix is the square of the number of perfect 
matchings. Therefore, by using the previous corollary to find an appropriate orientation of the edges of a graph $G$ to obtain a directed graph $D(G)$ with 
maximum number of $2$-factors, it is enough to focus on directed balanced graphs obtained from $G$.
 \end{rem}

Let $r$ and $s$ be positive integers and consider the complete bipartite graph $K_{r,s}$ with bipartition $\{X,Y\}$, where $|X|=r$ and $|Y|=s$. 
A \textit{bipartite tournament} of size $r$ by $s$ is any directed graph obtained from $K_{r,s}$ by assigning a direction to each of its edges. 
Also, denote by $\mathcal{D}_{r,s}$ the set of all bipartite tournaments of size $r$ by $s$.
Observe that $\mathcal{D}_{n}$, which was defined in \S1, is equal to $\mathcal{D}_{n,n}$.
Let
\begin{equation}\label{defrhors}
\rho(r,s)=\max_{BT\in \mathcal{D}_{r,s}} {\rm per}(A_{BT})={\rm per} (A_{BT^\star}).
\end{equation}
For two matrices $A$ and $B$, by $A\bigoplus B$ we mean the following matrix $$\left[\begin{array}{ll}A&0\\0&B\end{array}\right].$$
Denote by  $J_n$ the $n\times n$ matrix with every entry equal to $1$.
Recall Minc's upper bound conjecture for $(0,1)$ matrices \cite{minc}, which was proved by Bregman \cite[Theorems 4-5]{breg73}.
\begin{thm}\label{Minc-Bregthm}
 Let $A$ be a $(0,1)$-matrix of order $n$ with row sum vector $R=(r_1,\dots,r_n)$. Then 
$${\rm per}(A)\leq \prod_{i=1}^{n}(r_i!)^{\frac{1}{r_i}}$$ Moreover, equality holds if and only if 
$A$ is a bipartite adjacency matrix of disjoint union of complete bipartite graphs.  That is, one can permute the rows and columns of $A$
to obtain a direct sum $J_{n_1}\bigoplus \dots\bigoplus J_{n_m}$, where $n_1+\cdots+n_m=n$.
\end{thm}

For a positive integer $n,$ let $K_{n,n}$ be the complete  bipartite graph with partite sets $X=\{x_1,\dots,x_{n}\}$ and $Y=\{y_1,\dots,y_{n}\}$. 
Observe that the adjacency matrix of any bipartite tournament $BT\in \mathcal{D}_n$ is of the form
\begin{equation}\label{ABTform}
A_{BT}=\left[\begin{array}{ll}0&B\\(J_n-B)^T&0\end{array}\right]
\end{equation}
Here $B$ is any $n\times n$ matrix with $(0,1)$ entries.  Clearly, 
$${\rm per}(A_{BT})={\rm per}(B){\rm per}(J-B)^T={\rm per}(B){\rm per}(J_n-B).$$

Assume that $n$ is even and $BT_0(n,n)$ is the bipartite tournament obtained from $K_{n,n}$ by assigning an orientation in such a way that 
the first half of players of $X$ wins over the first half of  players of $Y$ and the second half  of  players of $X$ wins over the second half of  
players of $Y$. Also the first half of players of $Y$ wins over the second half of  players of $X$ and the second half  of players of $Y$ wins 
over the first half of players of $X$. Obviously, the adjacency matrix of $BT_0$ is of the form \eqref{ABTform} where 
$B=J_{\frac{n}{2}}\bigoplus J_{\frac{n}{2}}.$ Thus we have ${\rm per}(A_{BT_0})=(\frac{n}{2})!^4.$

The following theorem determines completely all elements of $\mathcal{D}_n$ for which the permanent is equal to $\rho(n,n)$, in the case of even $n$.
\begin{thm}
For any positive even integer $n$  $\rho(n,n)=(\frac{n}{2})!^4$.  Furthermore, $\rho(n,n)={\rm per} (A_{BT^\star})$ if and only if $BT^\star(n,n)\simeq BT_0(n,n).$ 
\end{thm}

\begin{proof}
{ Let $BT\in\mathcal{D}_n$.  So $A_{BT}$ is of the form \eqref{ABTform}.
 Let ${\rm deg}^+(x_i)=p_i$ and ${\rm deg}^-(x_i)=q_i$ for the vertex $x_i\in X$. Obviously, $p_i+q_i=n$, for every $i=1,\dots,n$.   
Clearly $(p_1,\dots,p_{n})$ and $(q_1,\dots,q_{n})$ are row sums of $B$ and $J-B$, respectively. So by the Minc-Bregman inequality  we have 
$${\rm per}(B){\rm per}(J_n-B)\leq \prod_{i=1}^{n}(p_i!)^{\frac{1}{p_i}}\prod_{i=1}^{n}(q_i!)^{\frac{1}{q_i}}.$$
Applying Corollary \ref{fundineqlem} for every $i=1,\ldots,n$ we find that
\[(p_i!)^{\frac{1}{p_i}}(q_i!)^{\frac{1}{q_i}}\le ((\frac{n}{2})!)^{\frac{4}{n}}.\]
Therefore
${\rm per}(A_{BT})=(p!)^{\frac{n}{p}}(q!)^{\frac{n}{q}}\leq  ((\frac{n}{2})!)^4$.  So $((\frac{n}{2})!)^4$ is an upper bound 
for $\rho(n,n)$. Again, by Corollary \ref{fundineqlem}, this upper bound can be achieved if and only if $p_i=q_i=\frac{n}{2}$ for $i=1,\ldots,n$.
Assume that $p_i=q_i=\frac{n}{2}$ for $i=1,\ldots,n$.   Minc-Bregman inequality yields that ${\rm per}(B)\le ((\frac{n}{2})!)^2$.
Equality holds if and only if $B$ is the bipartite adjacency matrix of disjoint union of complete bipartite graphs.
As the out degree of each vertex corresponding to $B$ is $\frac{n}{2}$, it follows that equality holds if and only if $B$ is the bipartite
adjacency matrix of the bipartite graph of $K_{\frac{n}{2}}\cup K_{\frac{n}{2}}$.  That is, one can permute the rows and the columns of $B$ to obtain
$J_{\frac{n}{2}}\bigoplus  J_{\frac{n}{2}}$.    Note that in this case ${\rm per}(J_n-B)=((\frac{n}{2})!)^2$.
 This completes the proof.}
 \end{proof}

Let $I_n$ be the identity matrix of order $n$. 
Denote by $D_n$ the number of derangements of $\{1,\ldots,n\}$.  That is $D_n$ is the number of all permutations $\sigma$ on  $\{1,\ldots,n\}$
such that $\sigma(i)\ne i, i=1,\ldots,n$.  It is well known that
$$D_n=n!\sum_{i=0}^n \frac{(-1)^i}{i!}\approx \frac{n!}{e} \quad.$$

\begin{pro}\label{comppernodd}  Let $n>1$ be odd and set $p=\frac{n-1}{2}$.  Denote by  $BT_1(n,n)$ and $BT_2(n,n)$ the tournaments corresponding to 
the matrices $A_{BT}$ of the form \eqref{ABTform}, where $B=J_p\oplus J_p\oplus J_1, and B=J_p\oplus (J_{p+1}-I_{p+1})$ respectively.  Then
\begin{equation}\label{comppernodd1} 
  {\rm per }( A_{BT_1(n,n)})=2p (p!)^4, \quad {\rm per }( A_{BT_2(n,n)})=(p+1)D_{p+1} (p!)^3.
\end{equation}
\end{pro}
\begin{proof}{  
Clearly,  $ {\rm per} (J_p)=p!$ and $ {\rm per} (J_{p+1}-I_{p+1})=D_{p+1}$.    Consider first ${\rm per}(A_{BT_1(n,n)})$.
Clearly, ${\rm per}(J_p\oplus J_p\oplus J_1)=(p!)^2$.   It is left to evaluate $ {\rm per} (C)$, where 
$$C=J_{2p+1}-(J_p\oplus J_p\oplus J_1)=\left[\begin{array}{ccc} 0& J_p&\1_p\\ J_p&0&\1_p\\\1_p^T&\1_p^T&0\end{array}\right].$$
Here $\1_p$ is a column vector whose all $p$ coordinates are $1$.  
Let $V_1:=\{1,\ldots,p\},V_2:=\{p+1,\ldots,n\}, V_2'=\{p+1,\ldots,n-1\}$.  For two subsets $U,W$ of $V_1\cup V_2$  we denote
by $C[U,W]$ the submatrix of $C$ with rows $U$ and columns $W$.

Let $W\subset V_1\cup V_2$ be of cardinality $p$.  Then  ${\rm per}(C[V_1,W])=0$
if $W\cap V_1\ne \emptyset$.   Laplace expansion of ${\rm per}(C)$ by the rows in $V_1$ yields:  
\begin{equation}\label{perCfor}
{\rm per}(C)=\sum_{i=p+1}^{n} {\rm per} (C[V_1,V_2\setminus\{i\}]){\rm per}(C[V_2,V_1\cup\{i\}]) .
\end{equation}
Clearly, ${\rm per} (C[V_1,V_2\setminus\{i\}])=p!$.
Observe next that 
$C[V_2,V_1\cup\{n\}]=J_{p+1}-(0\oplus J_1)$.  Expand ${\rm per}(C[V_2,V_1\cup\{n\}])$ by the last row to deduce ${\rm per}(C[V_2,V_1\cup\{n\}])=pp!$.
Consider now ${\rm per}(C[V_2,V_1\cup\{i\}]), i\in V_2'$.  Expand it by the last column to deduce that ${\rm per}(C[V_2,V_1\cup\{i\}])=p!$.  
This concludes the proof of the first equality in \eqref{comppernodd1}.

Consider now ${\rm per}(A_{BT_2(n,n)})$, where $A_{BT_2(n,n)}$ is of the form \eqref{ABTform} where $B=J_p\oplus (J_{p+1}-I_{p+1})$.
So ${\rm per}(B)=p!D_{p+1}$. 
It is left to evaluate $ {\rm per} (C)$, where 
$$C=J_{2p+1}-(J_p\oplus (J_{p+1}-I_{p+1}))=\left[\begin{array}{cc} 0& J_{p,p+1}\\ J_{p+1,p}&I_{p+1}\end{array}\right].$$
Here $J_{p,q}$ stands for the $p\times q$ matrix whose all entries are equal to $1$.  

As in the case of $BT_1(n,n)$ we deduce \eqref{perCfor}.
Clearly, ${\rm per} (C[V_1,V_2\setminus\{i\}])=p!$.  Expand  ${\rm per}(C[V_2,V_1\cup\{i\}])$
by the column $i$ to deduce that this permanent equals to $p!$.
This concludes the proof of the second equality in \eqref{comppernodd1}.
}
\end{proof}

Combine  Proposition \ref{comppernodd} and \eqref{muineq} to deduce for an odd $n>1$
\begin{equation}\label{rhobnd}
\max((n-1)(\frac{n-1}{2})!^4, \frac{n+1}{2} D_{ \frac{n+1}{2}}(\frac{n-1}{2})!^3)\leq \rho(n,n)\leq
 ((\frac{n-1}{2})!)^{\frac{2n}{n-1}}((\frac{n+1}{2})!)^{\frac{2n}{n+1}}.
\end{equation}
\\
Our computational work shows that the left hand side of the above inequality is sharp for $n=3,5,7$.  More precisely, for $n=5$ the maximal tournament is isomorphic
$BT_1(3,3)$.  For $n=3,7$ the maximal tournaments are either isomorphic to $BT_1(n,n)$ or $BT_2(n,n)$.  For an odd $n>7$ we have the inequality
$$(n-1)(\frac{n-1}{2})!^4< \frac{n+1}{2} D_{ \frac{n+1}{2}}(\frac{n-1}{2})!^3.$$
Stirling's formula yields that the ratio between the lower and the upper bounds in \eqref{rhobnd} for $n\gg 1$ is approximately $\frac{1}{e}$.
$$\frac{\frac{n+1}{2} D_{ \frac{n+1}{2}}(\frac{n-1}{2})!^3}{ ((\frac{n-1}{2})!)^{\frac{2n}{n-1}}((\frac{n+1}{2})!)^{\frac{2n}{n+1}}}\approx \frac{1}{e}.$$
\begin{con}
For an odd positive integer $n>7$, $\rho(n,n)= \frac{n+1}{2} D_{ \frac{n+1}{2}}(\frac{n-1}{2})!^3.$
\end{con}
\vspace{1cm}

\section{Some Computational Results and additional Observations}\label{comp}
Here we will first present some results on the maximum permanent and maximum number of perfect matchings for small tournaments and 
graphs with a small number of vertices, 
or with the number of edges close to the half of number of vertices 
\subsection{Small Tournaments}
Here we present some computational results on tournaments of order at most $10$. We generated all such a tournaments with an orderly 
algorithm and computed their permanent. Table \ref{tab1} shows the maximum value of the permanent function over tournaments of 
given order and the corresponding lower and upper bounds given by \eqref{lowtaubd} and
\begin{table}
\begin{center}
\begin{tabular}{|c|c|c|c|c|c|c|c|c|}
\hline
$n$ & 3 & 4 & 5 & 6 & 7 & 8 & 9 & 10 \\ \hline
$\tau(n)$ & 1 & 1 & 3 & 9 & 31 & 102 & 484 & 2350 \\ \hline
l.b.&1&1&1&1&1&13&17&255 \\ \hline
u.b.&1&2&5&17&62&272&1227&6602\\ \hline
\end{tabular}
\end{center}
 \caption{Values of $\tau(n)$ and bounds}
\label{tab1}
\end{table}

\subsection{Perfect matchings in small or sparse graphs}
For regular bipartite graphs the maximum number of perfect matchings, and matchings of all sizes, were computed in \cite{FKM}. Here we present 
the result of a similar computation for $n\leq 10$ vertices and each possible number of edges.  We also kept track of whether the extremal graphs 
were almost regular or not.

In Table 2 we present the maximum number of matchings in a graph with $n$ vertices and $m$ edges.  A number marked with a * marks 
a case where some of the extremal graphs are not almost regular and ** means that none of the extremal graphs are almost regular.

\begin{table}[htdp]
\begin{center}

{\footnotesize
\begin{tabular}{ccccc}
$m$ & $n=4$& $n=6$& $n=8$& $n=10$\\
 2 & 1 &   & \text{} & \text{} \\
 3 & 1 & 1 & \text{} & \text{} \\
 4 & 2 & 1 & 1 & \text{} \\
 5 & 2 & 2 & 1 & 1 \\
 6 & 3 & 2* & 2 & 1 \\
 7 & \text{} & 3* & 2* & 2 \\
 8 & \text{} & 4 & 4 & 2* \\
 9 & \text{} & 6 & 4* & 4 \\
 10 & \text{} & 6 & 6* & 4* \\
 11 & \text{} & 7 & 6* & 6* \\
 12 & \text{} & 8* & 9 & 8 \\
 13 & \text{} & 10 & 11 & 12 \\
 14 & \text{} & 12 & 14 & 12* \\
 15 & \text{} & 15 & 18 & 18 \\
 16 & \text{} & \text{} & 24 & 18* \\
 17 & \text{} & \text{} & 24 & 24** \\
 18 & \text{} & \text{} & 26 & 26 \\
 19 & \text{} & \text{} & 28* & 34 \\
 20 & \text{} & \text{} & 33 & 44 \\
 21 & \text{} & \text{} & 37 & 53 \\
 22 & \text{} & \text{} & 43 & 64 \\
 23 & \text{} & \text{} & 50 & 78 \\
 24 & \text{} & \text{} & 60 & 96 \\
 25 & \text{} & \text{} & 68 & 120 \\
 26 & \text{} & \text{} & 78 & 120 \\
 27 & \text{} & \text{} & 90 & 126 \\
 28 & \text{} & \text{} & 105 & 132* \\
 29 & \text{} & \text{} & \text{} & 145 \\
 30 & \text{} & \text{} & \text{} & 158** \\
 31 & \text{} & \text{} & \text{} & 178 \\
 32 & \text{} & \text{} & \text{} & 198 \\
 33 & \text{} & \text{} & \text{} & 225 \\
 34 & \text{} & \text{} & \text{} & 255 \\
 35 & \text{} & \text{} & \text{} & 295 \\
 36 & \text{} & \text{} & \text{} & 330 \\
 37 & \text{} & \text{} & \text{} & 372 \\
 38 & \text{} & \text{} & \text{} & 421 \\
 39 & \text{} & \text{} & \text{} & 478 \\
 40 & \text{} & \text{} & \text{} & 544 \\
 41 & \text{} & \text{} & \text{} & 604 \\
 42 & \text{} & \text{} & \text{} & 672 \\
 43 & \text{} & \text{} & \text{} & 750 \\
 44 & \text{} & \text{} & \text{} & 840 \\
 45 & \text{} & \text{} & \text{} & 945 \\
\end{tabular}
}
\caption{The maximum number of perfect matchings in graph with $n$ vertices and $m$ edges}	
\end{center}
\label{mtab}
\end{table}%

As we can see there are graphs with small  numbers of edges which are not almost regular, and this is a pattern that will persists for larger $n$ as well.  
In Figure \ref{fig1} and Figure \ref{fig2} 
we display the extremal graphs for $n=6$, $m=6$ and  $n=8$, $m=7$.   As we can see one can take the disjoint union of any of the three 6-vertex graphs 
and a 1-edge matching to form an extremal graph for $n=8$, $m=7$.  This pattern can be continued by using a larger matching:
\begin{thm}
	The extremal graphs on $6+2n$ vertices and $6+n$  edges are not all almost regular. 
\end{thm}
 \begin{proof}{Use Table 2 to assume that $n\ge 1$.  Note that an extremal $G=(V,E)$ must have a perfect matching $M=(V,E')$, where $|E'|=3+n$.  
Hence $|E\setminus E'|=3$.  
These $3$ edges $E\setminus E'$ connect at most $6$ vertices.  Hence $G$ contains at least $2n$ vertices whose degrees are $1$. 
Let $V'\subset V$ be a subset of cardinality $2n$ of vertices of $G$ where the degree of each vertex is $1$.  
Let $H=(V_1,E_1)$, where $V_1=V\setminus V'$ be the induced subgraph of $G$ by $V_1$.  So $|V_1|=6$ and $|E_1|\le 6$.
Note that $|E_1|$ must be even.
So ${\rm perfmat} (G)={\rm perfmat} (H)$.  If $|E_1|\le 4$ then Table 1 yields that ${\rm perfmat} (H)\le 1$.
Hence $G$ is extremal if and only if $|E_1|=6$ and $H$ must be one of the extremal graphs for $n=6,m=6$.
So $G(V')$ is a perfect matching on $V'$ and $G$ is a disjoint union of $H$ and $G(V')$.}
\end{proof}
We can build similar families for some other numbers of edges when $m$ is close to $n/2$. This sets out the region where  
$m=\frac{n}{2}+c_1$, where $c_1$ is constant or $c_1=o(n)$ as one where the almost regularity property of extremal graphs 
can be expected to have an irregular behavior.

\begin{figure}[H]
	\begin{center}
		\includegraphics[width=0.7\textwidth]{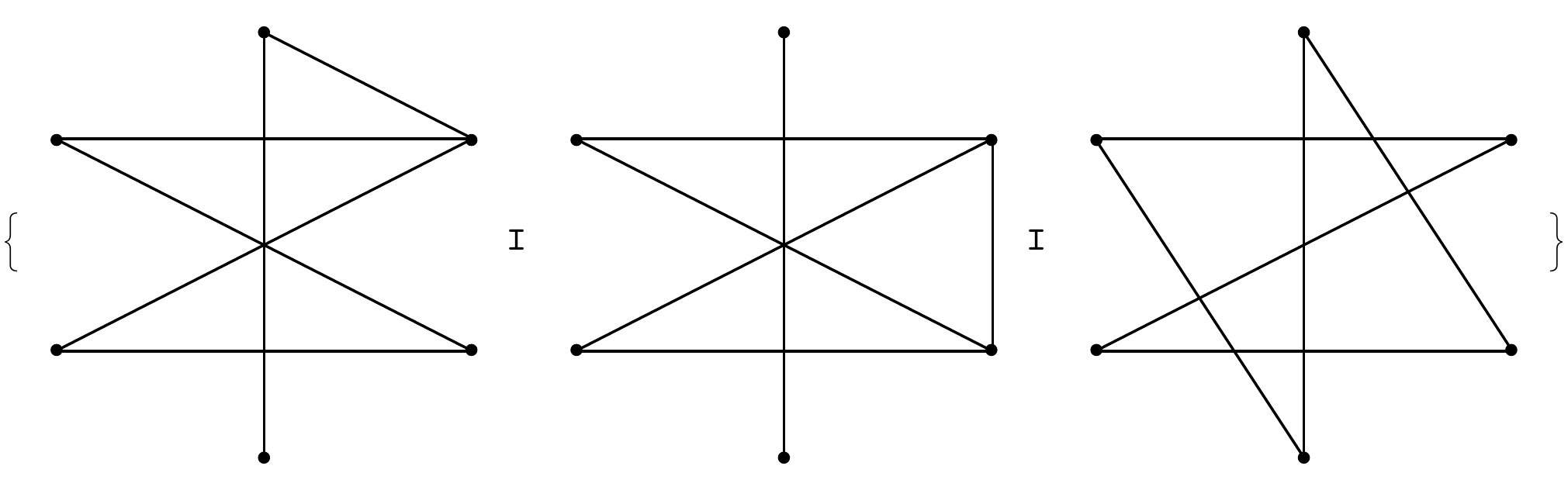}
		\caption{The graphs with maximum number of perfect matchings for $n=6$ and $m=6$ }
		\label{fig1}
	\end{center}
 \end{figure} 
\begin{figure}[H]
	\begin{center}
		\includegraphics[width=0.7\textwidth]{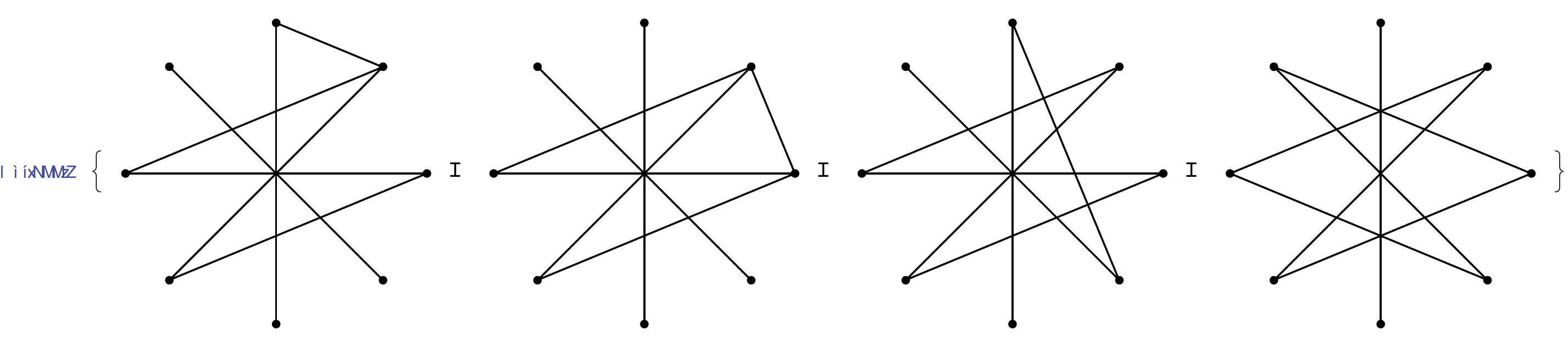}
		\caption{The graphs with maximum number of perfect matchings for $n=8$ and $m=7$ }
		\label{fig2}
	\end{center}
 \end{figure}

For $m=c_2{n \choose 2}$, i.e when the density of the graphs at hand is non-zero,  we would expect a smoother behaviour but as the example for 
$n=10$, $m=30$ shows there are still some surpirses here, at least for small $n$. That particular graph is displayed in Figure  \ref{fig3}.
\begin{figure}[H]
	\begin{center}
		\includegraphics[width=0.4\textwidth]{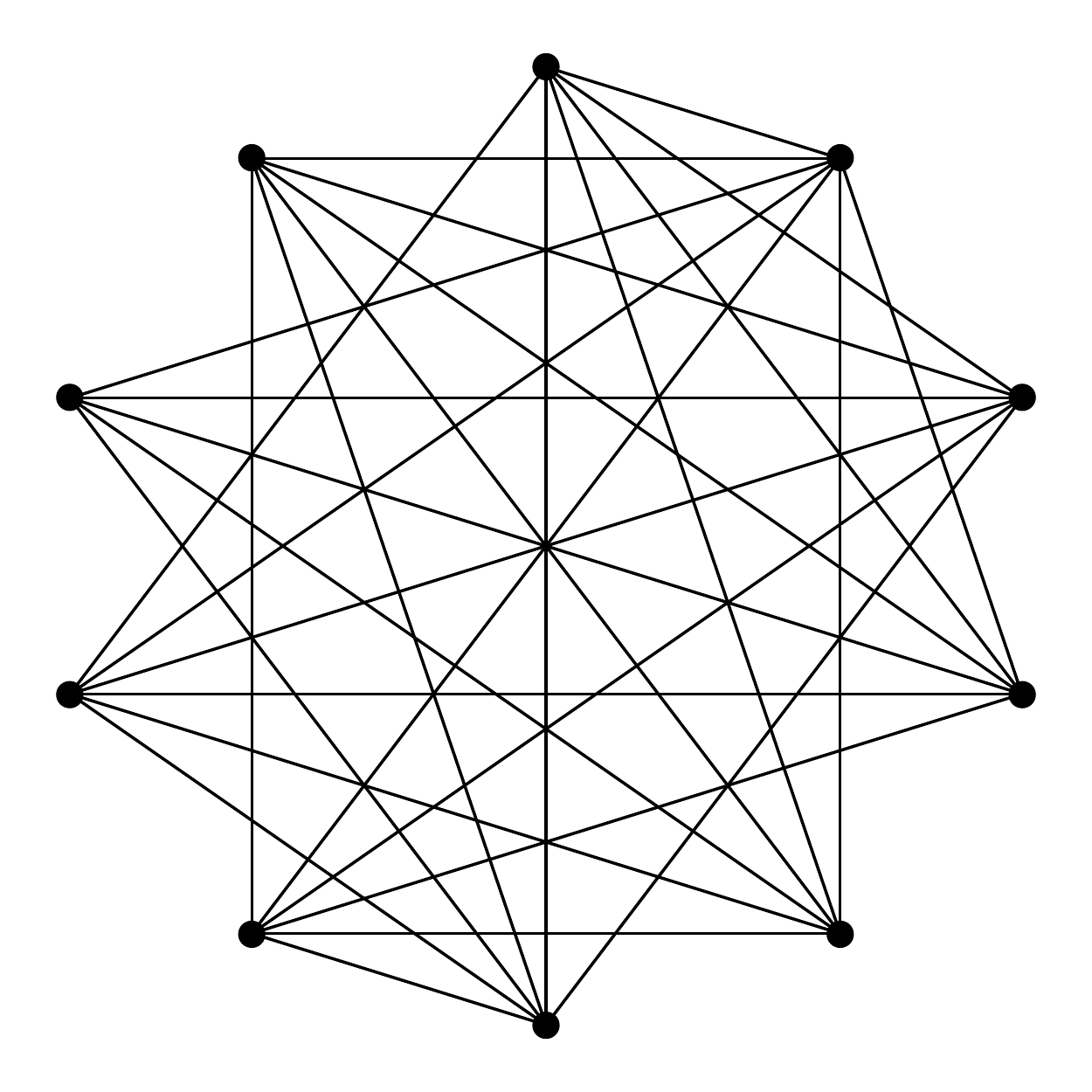}
		\caption{The graph with maximum number of perfect matchings for $n=10$ and $m=30$ }
		\label{fig3}
	\end{center}
 \end{figure} 
It is also interesting to plot the number of perfect matchings as a function of the number of edges. In Figure \ref{fig4} we have done this for $n=10$. As we can see there is a noticable change in the growth rate at $m=25$ and for the upper range of $m$ we have a convex function as well.
\begin{figure}[H]
	\begin{center}
		\includegraphics[width=0.7\textwidth]{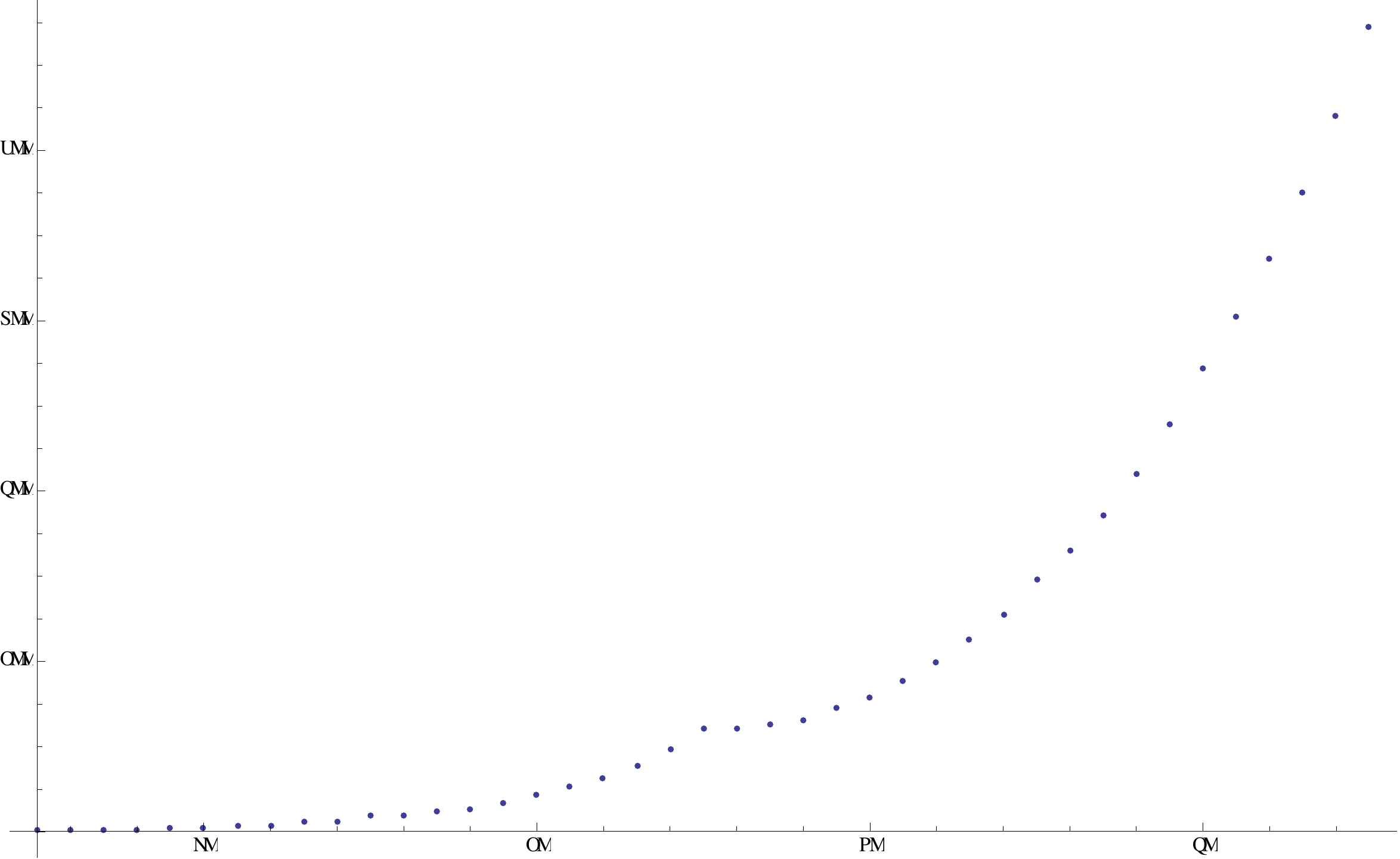}
		\caption{The maximum number of perfect matchings for $n=10$ as a function of $m$ }
		\label{fig4}
	\end{center}
 \end{figure} 

We conclude our paper with Figure \ref{fig5} that gives the ratio of the upper bound $\omega(2n,m)$, given in \eqref{defomegmn}, to the maximal number of matchings
$\mu(2n,m)$ for $n=10$ and $m\le 40$.

\begin{figure}[H]
	\begin{center}
		\includegraphics[width=0.7\textwidth]{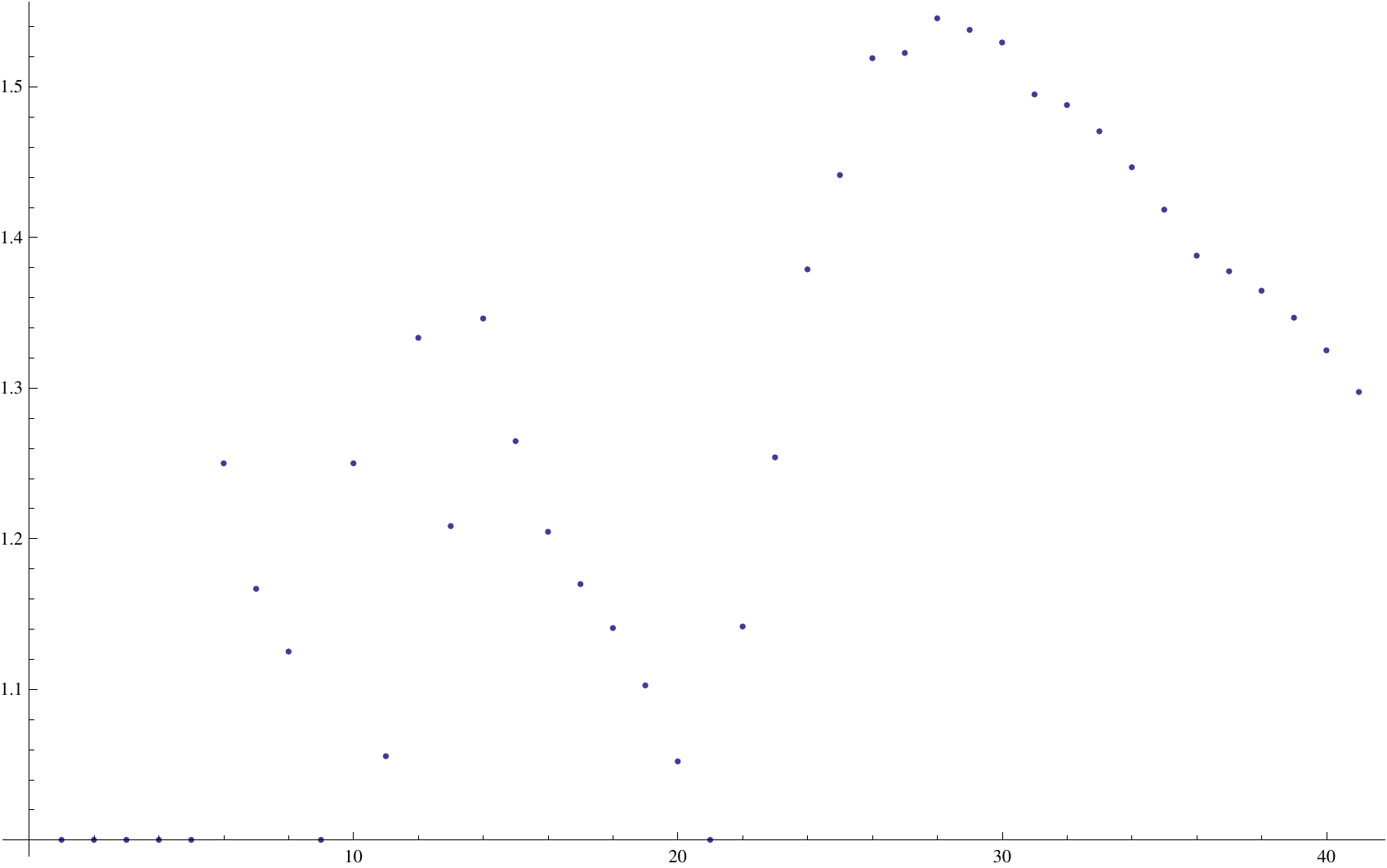}
		\caption{The ratio of the upper bound to the  maximum number of perfect matchings for $n=10$ as a function of $m$ }
		\label{fig5}
	\end{center}
 \end{figure} 

\noindent
\textbf{Acknowledgemens}

\noindent
 We thank Ian Wanless for useful comments.



\end{document}